\documentclass[reqno]{amsart}
\usepackage{amsfonts,amssymb}
\usepackage{enumerate}
\usepackage{verbatim}
\usepackage{mathrsfs}

\newcommand{\legendre}[2]{\genfrac{(}{)}{}{}{#1}{#2}}

\newtheorem{thm}{Theorem}%[section]
\newtheorem{lem}{Lemma}

\newtheorem{qu}{Question}
\theoremstyle{definition}

\theoremstyle{remark}

\author{Jing-Jing Huang}
\address{
Jing-Jing Huang: Department of Mathematics and Statistics, University of Nevada, Reno,
1664 N. Virginia St., Reno, NV 89557}
\email{jingjingh@unr.edu}
\dedicatory{}
\thanks{Research is supported by the UNR VPRI startup grant 1201-121-2479}

\begin{document}

\title
[Diophantine approximation on the parabola]
{Diophantine approximation on the parabola with non-monotonic approximation functions}

\begin{abstract}
We show that the parabola is of strong Khintchine type for convergence, which is the first result of its kind for curves. Moreover, Jarn\'{i}k type theorems are established in both the simultaneous and the dual settings, without monotonicity on the approximation function. To achieve the above, we prove a new counting result for the number of rational points with fixed denominators lying close to the parabola, which uses Burgess's bound on short character sums. 
\end{abstract}
\maketitle

\section{Introduction} \label{s1}

Let $k\in\mathbb{N}$, $\psi:\mathbb{N}\rightarrow[0,\infty)$ and
\begin{equation*}
\mathscr{S}_k(\psi):=\{\mathbf{x}\in\mathbb{R}^k:\exists^\infty{q}\in\mathbb{N}\textrm{ such that } \max_{1\le i\le k}\|{q}\cdot{x_i}\|<\psi(q)\},
\end{equation*}
where $\|\cdot\|$ denotes the distance to the nearest integer. Then it is an immediate consequence of the Borel-Cantelli lemma that $\mathscr{S}_k(\psi)$ has zero Lebesgue measure, when
\begin{equation}\label{e4}
\sum_{q=1}^\infty \psi^k(q)<\infty.
\end{equation}
This, together with a complementary divergence result under the additional assumption that $\psi$ is monotonically decreasing,  is a classical theorem of Khinchine. 

Modern development of the theory of metric diophantine approximation reveals that it is much more difficult to extend Khintchine's theorem to the context of a submanifold $\mathcal{M}$ of $\mathbb{R}^k$ of dimension less than $k$, since they have zero $k$-dimensional Lebesgue measure. More precisely, we say $\mathcal{M}$ is of \emph{Khintchine type for convergence} if for every monotonically decreasing function $\psi$ satisfying \eqref{e4}, the set $\mathscr{S}_k(\psi)\cap\mathcal{M}$ has measure zero with respect to the induced Lebesgue measure on $\mathcal{M}$. Following the terminology first introduced in \cite{Si}, we say $\mathcal{M}$ is of \emph{strong Khintchine type for convergence}, if the monotonicity condition on $\psi$ can be dropped in the above definition. 

There has been some substantial progress regarding Khintchine type manifolds (in both the divergence and the convergence cases), see \cite{bere, BDV, BeD, VV} and the references therein. However, somewhat surprisingly, we only know very little about strong Khintchine type manifolds for convergence. It is worth noting that Duffin and Schaeffer observed that, without monotonicity on $\psi$, the divergence of \eqref{e4} is not sufficient to guarantee that $\mathscr{S}_k(\psi)$  has full measure. Of course, one may assume the stronger condition $\sum_{q=1}^\infty \left(\frac{\phi(q)\psi(q)}q\right)^k=\infty$ and ask whether $\mathscr{S}_k(\psi)\cap \mathcal{M}$ has full induced measure on $\mathcal{M}$. This is analogous to the original Duffin-Schaeffer conjecture, therefore can be called the Duffin-Schaeffer conjecture for manifolds. To the best of our knowledge, there is no progress on this problem in the literature. 

It is shown in \cite{BVVZ, DRV, Si} that some classes of manifolds with various curvature and/or rank conditions are of strong Khintchine type for convergence. However, these conditions completely rule out curves. Indeed, it is remarked in \cite{Si} that 

\begin{center}{``\emph{It remains an open question whether nondegenerate planar curves (e.g. the standard parabola $\{(x, x^2):x\in\mathbb{R}\}$ are of strong Khinchin type for convergence. "}}
\end{center}

In this memoir, we answer the above question affirmatively for the standard parabola $\mathcal{P}:=\{(x, x^2):x\in [0,1]\}$, which has been studied as a pivotal case of non-degenerate planar curves. For instance, it has been shown to be extremal by Kubilyus \cite{K} and of Khintchine type for convergence by Bernik \cite{Ber}.
\begin{thm}\label{t1}
The parabola $\mathcal{P}$ is of strong Khintchine type for convergence.
\end{thm}

We actually prove a more general Jarn\'{i}k type theorem without assuming monotonicity on $\psi$. Let $\mathcal{H}^s$ be the standard Hausdorff $s$-measure.

\begin{thm}\label{t3}
Let $s\in(\frac{11}{13}, 1]$. If the series $\sum_{q=1}^\infty \psi(q)^{s+1}q^{1-s}<\infty$, then $\mathcal{H}^s(\mathscr{S}_2(\psi)\cap\mathcal{P})=0$.
\end{thm}
Some comments regarding the range of $s$ in the above theorem and the possibility of improving it can be found in \S\ref{s4}.

One may as well study the companion problem of dual approximation on planar curves. With the monotonicity condition on $\psi$, the general Hausdorff theory has been developed first for the parabola $\mathcal{P}$ \cite{hus} and then for all non-degenerate planar curves \cite{hua2}. In the special case when only Lebesgue measure is considered (sometimes called Groshev type theorem), it is already known that monotonicity is not needed for non-degenerate curves in $\mathbb{R}^n$ \cite{bere2, BD}. Without much more work, we are also able to remove the monotonicity condition for the Hausdorff theory of dual approximation on the parabola. Let
\begin{equation*}
\mathscr{A}_k(\psi):=\{\mathbf{x}\in\mathbb{R}^k:\exists^\infty \mathbf{q}\in\mathbb{Z}^k\textrm{ such that } \|\mathbf{q}\cdot\mathbf{x}\|<\psi(|\mathbf{q}|_\infty)\},
\end{equation*}
where $|\mathbf{q}|_\infty:=\max\{|q_1|, |q_2|, \ldots, |q_k|\}$.

\begin{thm}\label{t4}
Let $s\in(0,1]$. If the series $\sum_{q=1}^\infty \psi(q)^{s}q^{2-s}<\infty$, then $\mathcal{H}^s(\mathscr{A}_2(\psi)\cap\mathcal{P})=0$.
\end{thm}

It is well known that metric diophantine approximation on a manifold is closely related to counting rational points near the manifold. This is especially true in the case of simultaneous approximation, see \cite{bere, BDV, BZ, hua1, hua3, hua4, VV} for some recent advances.  However, it is usually much more difficult to obtain satisfactory estimates of the counting function without averaging over the denominator $q$, especially in the case of curves.  It is this additional averaging over $q$ that mandates the monotonicity of $\psi$ be assumed in order to prove Khintchine type theorems on the manifold.   We succeed in obtaining such a counting result for the parabola without averaging over $q$, which is nevertheless strong enough to prove the above theorems. 

Let $\delta\in(0,\frac12)$ and $$
A(q,\delta):=\sum_{\substack{a\le q\\\|a^2/q\|<\delta}}1.
$$
The function $A(q,\delta)$ naturally counts the number of rational points of fixed denominator $q$ that are lying close to $\mathcal{P}$. We obtain the following estimate of $A(q,\delta)$.

\begin{thm}\label{t2} Let $r$ be the largest integer such that $r^2|q$. Then for any $\varepsilon>0$
$$A(q,\delta)\ll \delta q+r^{1+\varepsilon}+\delta^{\frac12}q^{\frac{11}{16}+\varepsilon}$$
where the implied constant only depends on $\varepsilon$.
\end{thm}

Loosely speaking, Theorem \ref{t2} allows us to deduce that $A(q,\delta)$ is small for the vast majority of $q$ while it may occasionally become large when $q$ has a square factor almost as large as $q$. However, the latter case only happens very rarely. 

We remark in passing that the method within this paper can be adapted in a straightforward manner to treat the case of a quadratic polynomial with rational coefficients. The details are left to the interested reader. 

We will prove Theorem \ref{t2} first in \S\ref{s2}. Then using Theorem \ref{t2}, we will prove Theorem \ref{t3} in \S\ref{s3} and Theorem \ref{t4} in \S\ref{s4}. To prove Theorem \ref{t4} we will need to re-run the argument of \cite{hua2}, therefore we will only highlight the places where changes are needed and refer the reader to \cite{hua2} for the full structure of the proof.  In \S\ref{s5} we will discuss the limitation of our method and some future questions. 

\section{The proof of Theorem \ref{t2}}\label{s2}
Let $J=\left\lfloor \frac{1}{2\delta}\right\rfloor$. Recall the Fej\'{e}r kernel
$$
\mathcal{F}_J(x):=\sum_{j=-J}^J\frac{J-|j|}{J^2}e(jx),
$$
which satisfies $\mathcal{F}_J(x)\ge \frac4{\pi^2}$ when $\|x\|\le \delta$ and $\mathcal{F}_J(x)\ge 0$ for all $x$.
Then   we have
$$
A(q,\delta)\ll\sum_{a\le q} \mathcal{F}_J(a^2/q)\ll \delta q+D(q)
$$
where

$$D(q)\ll \sum_{j=1}^J\frac{J-j}{J^2}\sum_{a=1}^qe\left(\frac{ja^2}q\right).$$

Let 
$$G(j,q):=\sum_{a=1}^qe\left(\frac{ja^2}q\right)$$ and 
\[
\varepsilon_m=
\left\{
\begin{array}{ll}
1,&\text{when }m\equiv1\pmod{4}\\
i,&\text{when }m\equiv 3\pmod{4}.
\end{array}
\right.
\]

To evaluate these quadratic Gauss sums $G(j,q)$ precisely, we state the following lemma \cite[\S3.5]{IK}.
\begin{lem}\label{l1}
Suppose $(j,q)=1$. Then
\[
G(j,q)=
\left\{
\begin{array}{ll}
0,&\text{when } q\equiv2\pmod{4},\\
\varepsilon_q\legendre{j}{q}\sqrt{q},&\text{when } q \text{ is odd},\\
(1+i)\varepsilon_j^{-1}\legendre{q}{j}\sqrt{q},&\text{when } j \text{ is odd and }4|q.\\
\end{array}
\right.
\]
Here $\legendre{*}{*}$ is the Jacobi symbol.
\end{lem}
Note that $\legendre{*}{q}$ is a Dirichlet character modulo $q$ and $\legendre{q}{*}$ is a Dirichlet character of conductor $q'|4q$.

Clearly, these Gauss sums $G(j,q)$ are of size at most $\sqrt{2q}$. But this turns out to be insufficient for our purpose; we need to exploit as well the cancellation arising from the sign changes of those Jacobi symbols as $j$ varies. To that end, we state Burgess's landmark theorem on short character sums \cite{Bu}.
\begin{lem}\label{l2}
Let $\chi$ be a non-principal character modulo $q$. Then for any $\varepsilon>0$, we have
$$
\sum_{M<n\le M+N}\chi(n)\ll N^{\frac12}q^{\frac{3}{16}+\varepsilon}
$$
with the implied constant depending only on $\varepsilon$.
\end{lem}

We start by observing that
\begin{align*}
D(q)\ll &\sum_{j=1}^J\frac{J-j}{J^2}(j,q)G(j/(j,q),q/(j,q))\\
\ll& \sum_{d|q}d\sum_{\substack{j\le J\\(j,q)=d}}\frac{J-j}{J^2}G(j/d,q/d).
\end{align*}

For fixed $d|q$, let $q_1=q/d$. Then
\begin{lem}\label{l3}
$$
\sum_{\substack{j_1\le N\\(j_1,q_1)=1}}G(j_1,q_1)\ll
\left\{
\begin{array}{ll}
N\sqrt{q_1}&\text{if } q_1 \text{ is a square},\\
N^{\frac12}q_1^{\frac{11}{16}+\varepsilon}&\text{otherwise}.
\end{array}
\right.
$$
\end{lem}
\begin{proof}

Denote by $S$ the sum on the left side. If $q_1$ is a square, then by Lemma \ref{l1}
$$
S\ll\sum_{j_1\le N}\sqrt{q_1}=N\sqrt{q_1}.
$$

Next we treat the case when $q_1$ is not a square. Let $\chi_0$ be the principal character modulo 4 and $\chi_1$ be the quadratic character modulo 4. It is readily verified that 

$$
\varepsilon_m^{-1}=\frac{1-i}2\chi_{{}_{\scriptstyle0}}(m)+\frac{1+i}2\chi_{{}_{\scriptstyle1}}(m).
$$

If  $q_1$ is odd, let $\chi=\legendre{*}{q_1}$; if $q_1$ is even, let $\chi=\chi_{{}_{\scriptstyle0}}\legendre{q_1}{*}$ or $\chi=\chi_{{}_{\scriptstyle1}}\legendre{q_1}{*}$. Note that in any case, $\chi$ is always a non-principal character of modulus at most $4q_1$. Therefore, we may apply Burgess's bound (Lemma \ref{l2}) to the character sum $\sum_{j_1\le N}\chi(j_1)$ after another application of Lemma \ref{l1}, and obtain 
$$
S\ll N^{\frac12}q_1^{\frac3{16}+\varepsilon}\sqrt{q_1}=N^{\frac12}q_1^{\frac{11}{16}+\varepsilon}.
$$
\end{proof}

Hence by Lemma \ref{l3} and partial summation, we get
$$
\sum_{\substack{j_1\le J/d\\(j_1,q_1)=1}}\frac{J-j_1d}{J^2}G(j_1,q_1)\ll
\left\{
\begin{array}{ll}
\frac{\sqrt{q_1}}{d}&\text{if } q_1 \text{ is a square},\\
\frac{q_1^{\frac{11}{16}+\varepsilon}}{\sqrt{dJ}}&\text{otherwise}.
\end{array}
\right.
$$

 Then

\begin{align*}
D(q)\ll &\sum_{\substack{q_1|q\\q_1=\square}}\sqrt{q_1}+\sum_{\substack{q_1|q\\q_1\neq\square}}\sqrt{qJ^{-1}}q_1^{\frac{3}{16}+\varepsilon}\\
\ll&r^{1+\varepsilon}+\delta^{\frac12}q^{\frac{11}{16}+2\varepsilon}
\end{align*}
where, in the last line, we use the well known bounds $\sigma(r)=\sum_{r_1|r}r_1\ll r^{1+\varepsilon}$ and $d(q)=\sum_{q_1|q}1\ll q^{\varepsilon}$ on the arithmetic functions $\sigma(r)$ and $d(q)$ \cite[Chap. I.5]{Te}.

\section{The proof of theorem \ref{t3}}\label{s3}
By considering the auxiliary function $\hat\psi(q):=\max(\psi(q), q^{-5/8+\eta})$ for some $\eta\in\left(0,\frac58-\frac{2-s}{s+1}\right)$ if necessary, it is easily seen that there is no loss of generality in assuming that
\begin{equation}\label{e2}
\psi(q)\ge q^{-5/8+\eta}\quad \text{for all }q.
\end{equation}

We note that if $x\in[0,1]$ satisfies
$$
\left\{
\begin{array}{l}
|x-a/q|<\psi(q)/q\\
|x^2-b/q|<\psi(q)/q
\end{array}
\right.
$$ for some $1\le a,b\le q$, then
$$
\left|\frac{a^2}{q^2}-\frac{b}q\right|\le\left|\frac{a^2}{q^2}-x^2\right|+\left|x^2-\frac{b}q\right|<3\frac{\psi(q)}{q}.
$$
Hence, in order to show $\mathcal{H}^s(\mathscr{S}_2(\psi)\cap\mathcal{P})=0$, it suffices to show that the set
$$
\limsup_{q\to\infty}\bigcup_{\substack{a\le q\\ \|a^2/q\|<3\psi(q)}}\left(\frac{a}q-\frac{\psi(q)}q, \frac{a}q+\frac{\psi(q)}q\right)
$$
has zero Hausdorff $s$-measure.
This can be accomplished by the Hausdorff-Cantelli lemma, if we can show that
$$
\sum_{q=1}^\infty A(q,3\psi(q))\left(\frac{\psi(q)}q\right)^s<\infty.
$$

And indeed, by Theorem \ref{t2}, it follows that

\begin{align}
&\nonumber\sum_{q=1}^\infty A(q,3\psi(q))\left(\frac{\psi(q)}q\right)^s\\ \ll& \sum_{q=1}^\infty\psi(q)^{s+1}q^{1-s}+\sum_{q=1}^\infty r^{1+\varepsilon}\left(\frac{\psi(q)}q\right)^s+\sum_{q=1}^\infty\psi(q)^{\frac12}q^{\frac{11}{16}+\varepsilon}\left(\frac{\psi(q)}q\right)^s.\label{e1}\end{align}

The first series converges by assumption.  By H\"{o}lder's inequality, the second series above is bounded by
$$
\left(\sum_q (r^{1+\varepsilon}q^{-\frac{2s}{s+1}})^{s+1}\right)^\frac1{s+1}\left(\sum_q \psi(q)^{s+1}q^{1-s}\right)^\frac{s}{s+1}.
$$

On noting that
\begin{align*}
\sum_q (r^{1+\varepsilon}q^{-\frac{2s}{s+1}})^{s+1}&= \sum_{r=1}^\infty\sum_{t=1}^\infty |\mu(t)|r^{(1+\varepsilon)(s+1)}(r^2t)^{-2s}\\
&\le\sum_r r^{(1+\varepsilon)(s+1)-4s}\sum_t t^{-2s}
\end{align*}
and that $2s>22/13$ and $(1+\varepsilon)(s+1)-4s<-3/2$ for sufficiently small $\varepsilon$, we conclude the convergence of the second series in \eqref{e1}.

Taking $\varepsilon=\eta/2$, we deduce from \eqref{e2} that
$$
\psi(q)^\frac12q^{\frac{11}{16}+\varepsilon}\le \psi(q)q.
$$
So the third series in \eqref{e1} also converges. Thus we conclude that the series $\sum A(q,3\psi(q))\left(\frac{\psi(q)}q\right)^s$ converges and therefore complete the proof.

\section{Proof of Theorem \ref{t4}}\label{s4}
We will follow closely the proof in \cite{hua2}. The only places where averaging over $q$ is needed in \cite{hua2} are when \cite[Lemma 4]{hua2} and \cite[Lemma 5]{hua2} are applied. Here we replace them with the following variants of Theorem \ref{t2} in this paper. 
\begin{thm}\label{t5}
Let $r$ be the largest integer such that $r^2|q$, $\lambda\in\mathbb{Q}$, $I=[c, d]$ be an interval and $\alpha\in(0,1/2]$. Then for any $\varepsilon>0$ we have
$$\sum_{\substack{a/q\in I\\\|\lambda a^2/q\|<\delta}}1\ll \delta q+r^{1+\varepsilon}+\delta^{\frac12}q^{\frac{11}{16}+\varepsilon}$$
and
$$\sum_{\substack{a/q\in I\\\|\lambda a^2/q\|\ge\delta}}\left\|\lambda \frac{a^2}q\right\|^{-\alpha}\ll q+\delta^{-\alpha}r^{1+\varepsilon}.$$
\end{thm}

The proof of Theorem \ref{t2} can be easily adapted to prove the first estimate, while the second follows from the first by summing over  dyadic ranges.

Also note that the dual curve of $\mathcal{P}$ is $\{(x,-x^2/4): x\in[0,2]\}$. After making these changes, the bulk of the proof in \cite{hua2}  remains unchanged. The details are left to the reader.

\section{Further comments and questions}\label{s5}

If the conjectural bound 
\begin{equation}\label{e3}
\sum_{M<n\le M+N}\chi(n)\ll N^{\frac12}q^{\varepsilon}
\end{equation}
holds true for non-principal characters $\chi$ modulo $q$, one may verify that the third series in \eqref{e1} would be convergent for $s\in(1/2,1]$. But the range of $s$ in Theorem \ref{t3} can only be extended to $s\in(2/3,1]$, since when $s\le 2/3$ the second series in \eqref{e1} diverges.  This gives a measurement how far one might be able to improve Theorem \ref{t3}.
However \eqref{e3} is too far out of reach as it is closely related to the Generalized Riemann Hypothesis for Dirichlet $L$-functions, which demonstrates the depth of the problem under consideration in this paper.  Burgess's paper \cite{Bu} was published more than 54 years ago, but it still remains essentially the state of the art even today (at least for general modulus).  It is unlikely that the range of $s$ in Theorem \ref{t3} could be improved without beating Burgess's bound. 

Even for the case of parabola dealt with in this paper, we have already seen that some of the finest results in multiplicative number theory have to be used. It is conceivable that the general case of planar curves is significantly more difficult. It is not even clear, to the best of our knowledge,  what the general picture should look like. One may attempt to answer the following questions first.

\begin{qu}
For $\alpha\in\mathbb{R}\backslash\mathbb{Q}$, is the irrational parabola $y=\alpha x^2$ of strong Khintchine type for convergence?
\end{qu}

\begin{qu}
Is the unit circle $x^2+y^2=1$ of strong Khintchine type for convergence?
\end{qu}

\begin{qu}
Does there exist a non-degenerate planar curve which is not of strong Khintchine type for convergence?
\end{qu} 

It is hoped that this paper will attract more attention to these problems.

\proof[Acknowledgments]
The author would like to thank David Simmons for providing comments and Huixi Li for reading the manuscript.

\end{document}